\documentclass[12pt,a4paper]{article}
\usepackage{amsmath,amssymb,mathtools,amsthm}
\usepackage{geometry}
\newtheorem{thm}{Theorem}[section]
\newtheorem{pro}{Proposition}[section]
\newtheorem{lem}{Lemma}[section]

\title{{\normalsize\ \textbf{ZEROS OF CERTAIN WEAKLY HOLOMORPHIC MODULAR FORMS FOR THE FRICKE GROUP $\Gamma_0^+(3)$}}}
\author{Seiichi Hanamoto, Seiji Kuga}
\date{}
\begin{document}
\maketitle

\begin{abstract}
Let $M_k^!(\Gamma_0^+(3))$ be the space of weakly holomorphic modular forms of weight $k$ for the Fricke group of level $3$. We introduce a natural basis for $M_k^!(\Gamma_0^+(3))$ and prove that for almost all basis elements, all of their zeros in a fundamental domain lie on the circle centered at 0 with radius  $\frac{1}{\sqrt{3}}$.

\end{abstract}
\section{Introduction}
Rankin and Swinnerton-Dyer studied the location of the zeros of the Eisenstein series $E_k(z):=1-\frac{2k}{B_k}\sum_{n=1}^{\infty}\sigma_{k-1}(n)q^n$ of weight $k$ for the full modular group $SL_2(\mathbb{Z})$, where $q=e^{2\pi i z}$, $B_k$ is the $k$th Bernoulli number, and $\sigma_{k-1}(n)=\sum_{d|n}d^{k-1}$. They showed that the zeros of $E_k$ lie on the lower boundary arc of a fundamental domain for $SL_2(\mathbb{Z})$ [5]. Next, Duke and Jenkins showed that the zeros of certain natural basis for the space of weakly holomorphic modular forms for $SL_2(\mathbb{Z})$ lie on the arc of the fundamental domain [3]. Recently,  the locations of the zeros of weakly holomolphic modular forms for congruence subgroups or Fricke groups have been studied for various levels. In 2016, Choi and Im studied for the Fricke group of level $2$ and obtained a similar result in the case of $SL_2(\mathbb{Z})$ [1]. In this paper, we study the location of the zeros of natural basis for the space of weakly holomorphic modular forms for the Fricke group of level 3.

Let\ $k\in2\mathbb{Z}$, $\mathbb{H}:=\left\{z\in \mathbb{C}\ |\ {\rm Im}(z)>0\right\}$, $\Gamma_0(3):=\left\{\begin{pmatrix}a&b\\c&d\end{pmatrix}\in SL_2(\mathbb{Z})\ |\ c\equiv0\ (\bmod{3})\right\}$, $\Gamma_0^+(3):=\Gamma_0(3)\cup\begin{pmatrix}0&-\frac{1}{\sqrt{3}}\\ \sqrt{3}&0\end{pmatrix}\Gamma_0(3)$, and\ $q:=e^{2\pi iz}\ $for\ $z\in\mathbb{H}$. A holomorphic function $f$ on $\mathbb{H}$ is a weakly holomorphic modular form of weight $k$ with respect to $\Gamma_0^+(3)$ if $f$ satisfies
$$\begin{cases}f\left(\frac{az+b}{cz+d}\right)=(cz+d)^kf(z)$ for all $ \begin{pmatrix}a&b\\ c&d\end{pmatrix}\in \Gamma_0^+(3)$ and $z\in \mathbb{H}.\\ 
f$ has the $q$-expansion of the form $f(z)=\sum_{n\ge n_f}^\infty a_nq^n.\end{cases}$$

We define $f$ is holomorphic if $n_f\ge0$, a cusp form if $n_f\ge1$.
We denote the space of weakly holomorphic modular forms of weight $k$ on $\Gamma_0^+(3)$ by $M_k^!(\Gamma_0^+(3))$, the space of holomorphic modular forms by $M_k(\Gamma_0^+(3))$, and the space of cusp forms by $S_k(\Gamma_0^+(3))$. \\

Let
$$\mathbb{F}^+(3):=\left\{z\in \mathbb{H}\ |\ |z|\ge \frac{1}{\sqrt{3}},-\frac{1}{2}\le {\rm Re}(z)\le 0\right\}\cup \left\{z\in \mathbb{H}\ |\ |z|>\frac{1}{\sqrt{3}},0<{\rm Re}(z)<\frac{1}{2}\right\},$$
$$S:=\left\{\frac{1}{\sqrt{3}}e^{i\theta}\ |\ \frac{\pi}{2}\le \theta \le \frac{5\pi}{6}\right\}.$$
Then, $\mathbb{F}^+(3)$ is a fundamental domain of $\Gamma_0^+(3)$ [4, p. 694] and $S$ is the lower boundary arc of $\mathbb{F}^+(3)$.

Put $k=12\ell_k+r_k$ where $\ell_k\in \mathbb{Z}$, $r_k\in \{0,4,6,8,10,14\}$, and
\begin{center}
$\varepsilon_k:={\rm dim}\ S_{r_k}(\Gamma_0^+(3))=\begin{cases}0\ \ \ \textrm{if}\ r_k=0,4,6\\ 1\ \ \ \textrm{if}\ r_k=8,10,14\end{cases}$.\\
\end{center}
For each integer $m$ such that $m\ge-2\ell_k-\varepsilon_k$, there exists a unique weakly holomorphic modular form $f_{k,m}\in M_k^!(3)$ with $q$-expansion of the form
$$f_{k,m}(z)=q^{-m}+O(q^{2\ell_k+\varepsilon_k+1}).$$
Then, $\{f_{k,m}\}_{m\ge-2\ell_k-\varepsilon_k}$ form a natural basis of $M_k^!(3)$.\\

Let $\eta(z):=q^{\frac{1}{24}}\prod_{n=1}^\infty(1-q^n)$ be the Dedekind eta function. To construct the canonical basis for $M_k^!(3)$, we now define some weakly holomorphic modular forms as follows.
$$\begin{cases}\varDelta_3^+(z):=(\eta(z)\eta(3z))^{12}\in S_{12}(\Gamma_0^+(3)),\\E_k^+(z):=\frac{1}{1+3^\frac{k}{2}}(E_k(z)+3^\frac{k}{2}(3z))\in M_k(\Gamma_0^+(3)),\\j_3^+(z):=(\frac{\eta(z)}{\eta(3z)})^{12}+12+3^6(\frac{\eta(3z)}{\eta(z)})^{12}\in M_0^!(\Gamma_0^+(3)).\end{cases}$$
Furthermore, we define the holomorphic forms $\varDelta_{3,r_k}\in M_{r_k}(\Gamma_0^+(3))$ as follows.
$$\varDelta_{3,r_k}(z):=\begin{cases}1\ \ \ (r_k=0),\\E_4^+(z)\ \ \ (r_k=4),\\E_6^+(z)\ \ \ (r_k=6),\\\frac{41}{1728}(E_4^+(z)^2-E_8^+(z))\ \ \ (r_k=8),\\\frac{61}{432}(E_4^+(z)E_6^+(z)-E_{10}^+(z))\ \ \ (r_k=10),\\\frac{-22427}{272160}(E_6^+(z)E_8^+(z)-E_{14}^+(z))\ \ \ (r_k=14).\end{cases}$$
Referring to [2, Remark 2.2], $\varDelta_{3,r_k}$ can also be defined as the unique holomorphic form of weight $r_k$ with $q$-expansion of the form
$$\varDelta_{3,r_k}(z)=q^{\varepsilon_k}+O(q^{\varepsilon_k+1}).$$
Hence, $\varDelta_{3,r_k}\varDelta_{3,14-r_k}=\varDelta_{3,14}$ for any $r_k\in\{0,4,6,8,10,14\}$.\\
Then, we can construct the canonical basis $\{f_{k,m}\}_{m\ge-2\ell_k-\varepsilon_k}$ for $M_k^!(\Gamma_0^+(3))$ explicitly [2, Theorem 2.4].
\begin{equation}
f_{k,m}=(\varDelta_3^+)^{\ell_k}\varDelta_{3,r_k}F_f(j_3^+),
\end{equation}
where $F_f$ is the monic polynomial with integer coefficients of degree $2\ell_k+\varepsilon_k+m$ determined by $f$.\\
Now, our main theorem is following.
\begin{thm}
If $m\ge 18|\ell_k|+23$, then all of the zeros in $\mathbb{F}^+(3)$ of $f_{k,m}$ lie on $S$. 
\end{thm}
\section{Out line of the proof of Theorem 1.1}
In this section, we introduce four propositions to prove the main theorem and its proof. 
\begin{pro}
For $f(z)=\sum_{n\ge n_0}^\infty a(n)q^n \in M_k^!(\varGamma_0^+(3))$ with $a(n)$ is real for all n, then $e^{\frac{ik\theta}{2}}f(\frac{1}{\sqrt{3}}e^{i\theta})$ is real for all $\theta \in [\frac{\pi}{2},\frac{5\pi}{6}]$. In particular, $e^{\frac{ik\theta}{2}}f_{k,m}(\frac{1}{\sqrt{3}}e^{i\theta})$ is real for all $\theta \in [\frac{\pi}{2},\frac{5\pi}{6}]$.
\end{pro}
\begin{proof}
For all $z\in \mathbb{H}$, we note that
$$f\left(\frac{-1}{3z}\right)=(\sqrt{3}z)^kf(z),$$
and
$$\overline{f(z)}=\overline{\sum_{n\ge n_f} a(n)e^{2\pi i nz}}=\sum_{n\ge n_f}a(n)\overline{e^{2\pi i nz}}=\sum_{n\ge n_f}a(n)e^{2\pi i n(-\overline{z})}=f(-\overline{z}).$$
Put $z=\frac{1}{\sqrt{3}}e^{i\theta}\ (\frac{\pi}{2}\le\theta\le\frac{5\pi}{6})$, then $\frac{-1}{3z}=-\frac{1}{\sqrt{3}}e^{-i\theta}=-\overline{z}$. Hence 
$$\overline{f\left(\frac{1}{\sqrt{3}}e^{i\theta}\right)}=\overline{f(z)}=f(-\overline{z})=f\left(\frac{-1}{3z}\right)=(\sqrt{3}z)^kf(z)=e^{ik\theta}f\left(\frac{1}{\sqrt{3}}e^{i\theta}\right).$$
Thus, we obtain
$$\overline{e^{\frac{ik\theta}{2}}f\left(\frac{1}{\sqrt{3}}e^{i\theta}\right)}=e^{\frac{ik\theta}{2}}f\left(\frac{1}{\sqrt{3}}e^{i\theta}\right).$$
\end{proof}
\begin{pro}[valence formula]
Let $f\in M_k^!(\varGamma_0^+(3))$, which is not identically zero. We have
$$v_\infty (f)+\frac{1}{2}v_{\frac{i}{\sqrt{3}}}(f)+\frac{1}{6}v_{\rho_3}(f)+\sum_{\substack{\rho \ne \frac{i}{\sqrt{3}},\rho_{3}\\ \rho \in \mathbb{F}^+(3)}}v_\rho(f)=\frac{k}{6},$$
where $v_\rho(f)$ is the order of $f$ at $\rho$, and  $\rho_3=\frac{1}{\sqrt{3}}e^{\frac{5\pi i}{6}}$. {\rm (see [6])}
\end{pro}
\begin{pro}
For $f\in M_k^!(\varGamma_0^+(3))$, we have
\begin{align}
v_{\frac{i}{\sqrt{3}}}(f)&\ge s_k \ \ (s_k=0,1\ \ such\ that\ 2s_k\equiv k\mod 4),\notag\\
v_{\rho_3}(f)&\ge t_k \ \ (t_k=0,1,2,3,4,5\ \ such\ that\ -2t_k\equiv k\mod 12).\notag
\end{align}
\end{pro}
\begin{proof}
The claim is true for $k\ge4$ and $f\in M_k(\Gamma_0^+(3))$ [4, Proposition B.1]. For general $k\in2\mathbb{Z}$ and $f\in M_k^!(\Gamma_0^+(3))$, $(1)$ implies that $f$ can be written as the product of $\varDelta_{3,r_k}\in M_{r_k}(\Gamma_0^+(3))$ and some $f_0\in M_{12\ell_k}^!(\Gamma_0^+(3))$. Hence, the claim is also true for $f$.
\end{proof}
\begin{pro}
\ \\
{\rm (a)}\ \ For all $\theta \in[\frac{\pi}{2},\frac{23}{10}]$, if $m\ge 9|\ell_k|-2\ell_k+18$,
$$\left|e^{-2\pi m\frac{1}{\sqrt{3}}\cos{\theta}}e^{\frac{ik\theta}{2}}f_{k,m}\left(\frac{1}{\sqrt{3}}e^{i\theta}\right)-2\cos\left({\frac{k\theta}{2}-2\pi m\frac{1}{\sqrt{3}}\cos{\theta}}\right)\right|<1.9674.$$
{\rm (b)}\ \ For all $\theta \in[\frac{23}{10},\frac{5\pi}{6}-\frac{12}{25m}]$, if $m\ge 18|\ell_k|+23$,
$$\left|e^{-2\pi m\frac{1}{\sqrt{3}}\cos{\theta}}e^{\frac{ik\theta}{2}}f_{k,m}\left(\frac{1}{\sqrt{3}}e^{i\theta}\right)-2\cos\left({\frac{k\theta}{2}-2\pi m\frac{1}{\sqrt{3}}\cos{\theta}}\right)\right|<0.99728.$$
\end{pro}
\ \\
We will prove Proposition 2.4 in §4 and §5.\\
\\
\\
{\large(The proof of Theorem 1.1)}
\begin{proof}

Put $h(\theta):=e^{-2\pi m\frac{1}{\sqrt{3}}\cos{\theta}}e^{\frac{ik\theta}{2}}f_{k,m}(\frac{1}{\sqrt{3}}e^{i\theta})$ and $\alpha(\theta):=\frac{k\theta}{2}-2\pi m\frac{1}{\sqrt{3}}\cos{\theta}\ (\frac{\pi}{2}\le \theta \le \frac{5\pi}{6}-\frac{12
}{25m})$. Then, Proposition 2.1 implies $h(\theta)$ is a real valued function. $\alpha(\theta)$ is monotonically increasing on the interval $[\frac{\pi}{2},\frac{5\pi}{6}-\frac{12
}{25m}]$ and the image of $\alpha$ contains the interval $[\frac{k}{4}\pi,(\frac{5k}{12}+m-\frac{1}{3})\pi]$. By Proposition 2.4,
\begin{align}
&\theta \in\left[\frac{\pi}{2},\frac{23}{10}\right]\ \ \Longrightarrow \ \ |h(\theta)-2\cos{\alpha(\theta)}|<1.9674,\\
&\theta \in\left[\frac{23}{10},\frac{5\pi}{6}-\frac{12}{25m}\right]\ \ \Longrightarrow \ \ |h(\theta)-2\cos{\alpha(\theta)}|<0.99728.
\end{align}

Let $\lfloor \cdot \rfloor$ be the floor function and $\lceil \cdot \rceil$ be the ceiling function. When $\alpha(\theta)$ moves from $\lceil\frac{k}{4}\rceil \pi$ to $(\lfloor \frac{5k}{12}\rfloor+m-1)\pi$, there are exactly $\lfloor \frac{5k}{12}\rfloor+m-\lceil\frac{k}{4}\rceil$ values of $\theta$ in the interval $\alpha^{-1}(\left[\lceil\frac{k}{4}\rceil \pi,(\lfloor \frac{5k}{12}\rfloor+m-1)\pi\right])$ where $2\cos{\alpha(\theta)}$ has absolute value $2$, alternating between $2$ and $-2$ as $\theta$ increase. $(2),(3)$, and the intermediate value theorem imply that $h(\theta)$ must have at least $\lfloor \frac{5k}{12}\rfloor+m-\lceil\frac{k}{4}\rceil-1$ distinct zeros as $\theta$ moves through in the interval $\alpha^{-1}((\lceil\frac{k}{4}\rceil \pi,(\lfloor \frac{5k}{12}\rfloor+m-1)\pi))$.

When $\alpha(\theta)$ moves from $(\lfloor \frac{5k}{12}\rfloor+m-\frac{2}{3})\pi$ to $(\lfloor \frac{5k}{12}\rfloor+m-\frac{1}{3})\pi$, $2\cos{\alpha(\theta)}$ moves from $\pm1$ to $\mp1$. $(3)$ and the intermediate value theorem imply that $h(\theta)$ must have a zero in the interval $\alpha^{-1}((\lfloor \frac{5k}{12}\rfloor+m-\frac{2}{3})\pi,(\lfloor \frac{5k}{12}\rfloor+m-\frac{1}{3})\pi))$.\\
Thus $h(\theta)$ must have at least $\lfloor \frac{5k}{12}\rfloor+m-\lceil\frac{k}{4}\rceil$ distinct zeros as $\theta$ moves through in the interval $(\frac{\pi}{2},\frac{5\pi}{6}-\frac{12}{25m})$.\\

We note that $\lfloor \frac{5k}{12}\rfloor+m-\lceil\frac{k}{4}\rceil=2\ell_k+m+\varepsilon_k$, the main theorem follows from Propositions 2.2, 2.3, and a simple calculation.
\end{proof}

\section{Integral formula for $f_{k,m}$}
Referring to [2, p. 756] we have the following integral formula of $f_{k,m}$ which plays an important role in the proof of Proposition 1.2.
\begin{pro}
Let $f_k:=f_{k,-2\ell_k-\varepsilon_k}=(\varDelta_3^+)^{\ell_k}\varDelta_{3,r_k}$, we have
$$f_{k,m}(z)=\frac{1}{2\pi i}\oint_C\frac{f_k(z)f_{2-k}(\tau){q'}^{-m-1}}{j_3^+(\tau)-j_3^+(z)}dq',$$
where $q':=e^{2\pi i \tau}$ and $C$ is the circle centered at 0 in the $q'$-plane with a sufficiently small radius.
\end{pro}
Since $q'=e^{2\pi i \tau}$, $\frac{dq'}{d\tau}=2\pi iq'$, $\ell_{2-k}=-\ell_k-1$, $r_{2-k}=14-r_k$, and $\frac{dj_3^+}{d\tau}=-2\pi i\frac{\varDelta_{3,14}}{\varDelta_3^+}$, we can rearrange the integral formula of $f_{k,m}$ as follows.
\begin{align}
f_{k,m}(z) &=\int_{-\frac{1}{2}+iA}^{\frac{1}{2}+iA}\frac{f_k(z)f_{2-k}(\tau)e^{-2\pi im\tau}}{j_3^+(\tau)-j_3^+(z)}d\tau\notag\\
\
&=\int_{-\frac{1}{2}+iA}^{\frac{1}{2}+iA}\frac{\{\varDelta_3^+(z)^{\ell_k}\varDelta_{3,r_k}(z)\}\{\varDelta_3^+(\tau)^{-\ell_k-1}\varDelta_{3,14-r_k}(\tau)\}e^{-2\pi im\tau}}{j_3^+(\tau)-j_3^+(z)}d\tau\notag\\
&=\int_{-\frac{1}{2}+iA}^{\frac{1}{2}+iA}\frac{e^{-2\pi im\tau}\varDelta_3^+(z)^{\ell_k}\varDelta_{3,r_k}(z)\varDelta_{3,14-r_k}(\tau)}{\varDelta_3^+(\tau)^{\ell_k+1}(j_3^+(\tau)-j_3^+(z))}d\tau\notag\\
&=\int_{-\frac{1}{2}+iA}^{\frac{1}{2}+iA}e^{-2\pi im\tau}\frac{\varDelta_3^+(z)^{\ell_k}\varDelta_{3,r_k}(z)}{\varDelta_3^+(\tau)^{\ell_k}\varDelta_{3,r_k}(\tau)}\cdot\frac{\varDelta_{3,14}(\tau)}{\varDelta_3^+(\tau)}\cdot\frac{1}{j_3^+(\tau)-j_3^+(z)}d\tau\notag\\
&=\int_{-\frac{1}{2}+iA}^{\frac{1}{2}+iA}\frac{e^{-2\pi im\tau}}{-2\pi i}\frac{\varDelta_3^+(z)^{\ell_k}\varDelta_{3,r_k}(z)}{\varDelta_3^+(\tau)^{\ell_k}\varDelta_{3,r_k}(\tau)}\frac{\frac{d}{d\tau}(j_3^+(\tau)-j_3^+(z))}{j_3^+(\tau)-j_3^+(z)}d\tau,
\end{align}
where $A>0$ is sufficiently large. We write briefly\\
\begin{align}
G(\tau,z):&=e^{-2\pi im\tau}\frac{\varDelta_3^+(z)^{\ell_k}\varDelta_{3,r_k}(z)}{\varDelta_3^+(\tau)^{\ell_k}\varDelta_{3,r_k}(\tau)}\cdot\frac{\varDelta_{3,14}(\tau)}{\varDelta_3^+(\tau)}\cdot\frac{1}{j_3^+(\tau)-j_3^+(z)}\notag\\
&=\frac{e^{-2\pi im\tau}}{-2\pi i}\frac{\varDelta_3^+(z)^{\ell_k}\varDelta_{3,r_k}(z)}{\varDelta_3^+(\tau)^{\ell_k}\varDelta_{3,r_k}(\tau)}\frac{\frac{d}{d\tau}(j_3^+(\tau)-j_3^+(z))}{j_3^+(\tau)-j_3^+(z)}.
\end{align}

\section{The case of $\frac{\pi}{2}<\theta<\frac{23}{10}$}

Put $z=\frac{1}{\sqrt{3}}e^{i\theta}$ $(\frac{\pi}{2}<\theta<\frac{23}{10})$ and $A'=0.35$. We move the countor of integration given in $(4)$ downward to a height $A'$. As we do so, each pole $\tau_0$ of $G(\tau,z)$ in the region defined by
\begin{center}
$-\frac{1}{2}<{\rm Re}(\tau)<\frac{1}{2}$ and $A'<{\rm Im}(\tau)<A$
\end{center}
will contribute a term $2\pi i\cdot{\rm Res}_{\tau=\tau_0}G(\ \cdot\ ,z)$ to the equation. The pole of $G(\ \cdot\ ,z)$ occurs only when $\tau=z$ or $\frac{-1}{3z}$ which are equivalent to $z$ under the action of $\Gamma_0^+(3)$. Then the residue theorem yields
\begin{align}
\int_{-\frac{1}{2}+iA'}^{\frac{1}{2}+iA'}G(\tau,z)d\tau&=f_{k,m}(z)-\int_{-\frac{1}{2}+iA'}^{-\frac{1}{2}+iA}G(\tau,z)d\tau+\int_{\frac{1}{2}+iA'}^{\frac{1}{2}+iA}G(\tau,z)d\tau+2\pi i\sum_{\tau=z,\frac{-1}{3z}}{\rm Res}_\tau G(\ \cdot\ ,z)\notag\\
&=f_{k,m}(z)+2\pi i\sum_{\tau=z,\frac{-1}{3z}}{\rm Res}_\tau G(\ \cdot\ ,z).
\end{align}
By using (5), we can calculate  ${\rm Res}_z G(\ \cdot\ ,z)$ and ${\rm Res}_{\frac{-1}{3z}}G(\ \cdot\ ,z)$ directly
\begin{align}
{\rm Res}_z G(\ \cdot\ ,z)&=\frac{e^{-2\pi imz}}{-2\pi i}\notag\\
&=\frac{e^{2\pi m\frac{1}{\sqrt{3}}\sin\theta}e^{-2\pi im\frac{1}{\sqrt{3}}\cos\theta}}{-2\pi i},\\
{\rm Res}_{\frac{-1}{3z}}G(\ \cdot\ ,z)&=\frac{e^{-2\pi i m\frac{-1}{3z}}}{-2\pi i(\sqrt{3}z)^k}\notag\\
&=\frac{e^{2\pi m\frac{1}{\sqrt{3}}\sin\theta}e^{2\pi im\frac{1}{\sqrt{3}}\cos\theta}e^{-ik\theta}}{-2\pi i}.
\end{align}
By substituting (7) and (8) for (6), we have
$$\int_{-\frac{1}{2}+iA'}^{\frac{1}{2}+iA'}G(\tau,z)d\tau=f_{k,m}(z)-(e^{2\pi m\frac{1}{\sqrt{3}}\sin\theta}e^{-2\pi im\frac{1}{\sqrt{3}}\cos\theta}+e^{2\pi m\frac{1}{\sqrt{3}}\sin\theta}e^{2\pi im\frac{1}{\sqrt{3}}\cos\theta}e^{-ik\theta}).$$
Multiplying the both side by $e^{-2\pi m\frac{1}{\sqrt{3}}\sin\theta}e^{\frac{ik\theta}{2}}$, we have
\begin{align}
e^{-2\pi m\frac{1}{\sqrt{3}}\sin\theta}e^{\frac{ik\theta}{2}}f_{k,m}\left(\frac{1}{\sqrt{3}}e^{i\theta}\right)-2\cos&\left(\frac{k\theta}{2}-2\pi m\frac{1}{\sqrt{3}}\cos\theta\right)\notag\\
&=e^{-2\pi m\frac{1}{\sqrt{3}}\sin\theta}e^{\frac{ik\theta}{2}}\int_{-\frac{1}{2}+iA'}^{\frac{1}{2}+iA'}G(\tau,z)d\tau.
\end{align}
Hence, the absolute value of the left hand side of (9) is bounded above by
\begin{align}
e^{-2\pi m\frac{1}{\sqrt{3}}\sin{\theta}}&\int_{-\frac{1}{2}+iA'}^{\frac{1}{2}+iA'}|G(\tau,z)|d|\tau|\notag\\ 
&\le e^{-2\pi m\frac{1}{\sqrt{3}}\sin{\theta}}\int_{-\frac{1}{2}+iA'}^{\frac{1}{2}+iA'}e^{2\pi m{\rm Im}(\tau)}\frac{|\varDelta_3^+(z)|^{\ell_k}}{|\varDelta_3^+(\tau)|^{\ell_k}}\cdot\frac{|\varDelta_{3,r_k}(z)||\varDelta_{3,14-r_k}(\tau)|}{|\varDelta_3^+(\tau)||j_3^+(\tau)-j_3^+(z)|}d|\tau|\notag\\ 
&\le e^{-2\pi m(\frac{1}{\sqrt{3}}\sin{(\frac{23}{10})}-0.35)}\int_{-\frac{1}{2}+iA'}^{\frac{1}{2}+iA'}\frac{|\varDelta_3^+(z)|^{\ell_k}}{|\varDelta_3^+(\tau)|^{\ell_k}}\cdot\frac{|\varDelta_{3,r_k}(z)||\varDelta_{3,14-r_k}(\tau)|}{|\varDelta_3^+(\tau)||j_3^+(\tau)-j_3^+(z)|}d|\tau|.
\end{align}
To prove the Proposition 2.4 (a), we show the next lemma.
\begin{lem}
Let $z=\frac{1}{\sqrt{3}}e^{i \theta}$ with $\frac{\pi}{2}\le \theta <\frac{23}{10}$ and $\tau =x+0.35i$ with $-\frac{1}{2}\le x\le \frac{1}{2}$,\\
{\rm (a)} $2.8964\times10^{-4}<|\Delta_3^+(z)|<1.0258\times10^{-2}$.\\
{\rm (b)} $4.3086\times10^{-4}<|\Delta_3^+(\tau)|<5.0415\times10^{-2}$.\\
{\rm (c)} $|j_3^+(\tau)-j_3^+(z)|>0.41095$.\\
{\rm (d)} For $r_k=0,4,6,8,10,14,\ |\Delta_{3,r_k}(z)||\Delta_{3,14-r_k}(\tau)|<3.1448$.\\
\end{lem}
\begin{proof}
We have used M{\scriptsize ATHEMATICA} 11 for the following computation.

(a) and (b)
$$\Delta_3^+(z)=e^{4\pi iz}\prod_{n=1}^{\infty}(1-e^{2\pi inz})^{12}\prod_{n=1}^{\infty}(1-e^{6\pi i nz})^{12}.$$
By Euler's pentagonal number theorem,
\begin{equation}
\left|\prod_{n=1}^{\infty}(1-e^{2\pi inz})\right|=\left|\sum_{n\in\mathbb{Z}}(-1)^ne^{2\pi iz\cdot\frac{3n^2-n}{2}}\right|\le\sum_{n\in\mathbb{Z}}e^{-\pi{\rm Im}(z)(3n^2-n)}\notag
\end{equation}
and by the triangle inequality,
\begin{equation}
\left|\prod_{n=1}^{\infty}(1-e^{2\pi inz})\right|=\prod_{n=1}^{\infty}|1-e^{2\pi inz}|\ge\prod_{n=1}^{\infty}|1-|e^{2\pi inz}||=\prod_{n=1}^{\infty}(1-e^{-2\pi n{\rm Im}(z)}\notag
\end{equation}
for all $z\in\mathbb{H}$.\\
Hence, for $z=\frac{1}{\sqrt{3}}e^{i\theta}$ with $\frac{\pi}{2}<\theta<\frac{23}{10}$, we have
\begin{align}
|\Delta_3^+(z)|&\le e^{-4\pi {\rm Im}(z)}\left(\sum_{n\in\mathbb{Z}}e^{-\pi{\rm Im}(z)(3n^2-n)}\right)^{12}\left(\sum_{n\in\mathbb{Z}}e^{-3\pi{\rm Im}(z)(3n^2-n)}\right)^{12}\notag\\
&\le e^{-\frac{4\pi}{\sqrt{3}}\sin{(\frac{23}{10})}}\left(\sum_{n\in\mathbb{Z}}e^{-\frac{\pi}{\sqrt{3}}\sin{(\frac{23}{10})}(3n^2-n)}\right)^{12}\left(\sum_{n\in\mathbb{Z}}e^{-\sqrt{3}\pi\sin{(\frac{23}{10})}(3n^2-n)}\right)^{12}\notag\\
&<1.0258\times10^{-2}\notag
\end{align}
and
\begin{align}
|\Delta_3^+(z)|&\ge e^{-4\pi {\rm Im}(z)}\prod_{n=1}^{\infty}(1-e^{-2\pi n{\rm Im}(z)})^{12}\prod_{n=1}^{\infty}(1-e^{-6\pi n{\rm Im}(z)})^{12}\notag\\
&\ge e^{-\frac{4\pi}{\sqrt{3}}}\prod_{n=1}^{\infty}(1-e^{-\frac{2}{\sqrt{3}}\pi\sin{(\frac{23}{10})}n})^{12}\prod_{n=1}^{\infty}(1-e^{-2\sqrt{3}\pi\sin{(\frac{23}{10})} n})^{12}\notag\\
&>2.8964\times10^{-4}\notag
\end{align}
Also for $\tau=x+0.35i$ with $-\frac{1}{2}\le x\le\frac{1}{2}$, we have
\begin{align}
|\Delta_3^+(\tau)|&\le e^{-4\pi {\rm Im}(\tau)}\left(\sum_{n\in\mathbb{Z}}e^{-\pi{\rm Im}(\tau)(3n^2-n)}\right)^{12}\left(\sum_{n\in\mathbb{Z}}e^{-3\pi{\rm Im}(\tau)(3n^2-n)}\right)^{12}\notag\\
&=e^{-1.4\pi}\left(\sum_{n\in\mathbb{Z}}e^{-0.35\pi(3n^2-n)}\right)^{12}\left(\sum_{n\in\mathbb{Z}}e^{-1.05\pi(3n^2-n)}\right)^{12}\notag\\
&<5.0415\times10^{-2}\notag
\end{align}
and
\begin{align}
|\Delta_3^+(\tau)|&\ge e^{-4\pi {\rm Im}(\tau)}\prod_{n=1}^{\infty}(1-e^{-2\pi n{\rm Im}(\tau)})^{12}\prod_{n=1}^{\infty}(1-e^{-6\pi n{\rm Im}(\tau)})^{12}\notag\\
&=e^{-1.4\pi}\prod_{n=1}^{\infty}(1-e^{-0.7\pi n})^{12}\prod_{n=1}^{\infty}(1-e^{-2.1\pi n})^{12}\notag\\
&>4.3086\times10^{-4}.\notag
\end{align}

(c)
$$|j_3^+(\tau)-j_3^+(z)|\ge\left|j_3^+(\tau)-j_3^+\left(\frac{1}{\sqrt{3}}e^{\frac{\pi i}{2}}\right)\right|-\left|j_3^+\left(\frac{1}{\sqrt{3}}e^{\frac{\pi i}{2}}\right)-j_3^+(z)\right|.$$
First, we estimate $|j_3^+(\tau)-j_3^+(\frac{1}{\sqrt{3}}e^{\frac{\pi i}{2}})|$.\\
Put $\omega=e^{\frac{\pi i}{3}}$,
\begin{align}
j_3^+(\tau)-j_3^+\left(\frac{1}{\sqrt{3}}e^{\frac{\pi i}{2}}\right)&=j_3^+(\tau)-66\notag\\
&=\left(\frac{\eta(\tau)}{\eta(3\tau)}\right)^{12}+3^6\left(\frac{\eta(3\tau)}{\eta(\tau)}\right)^{12}-54\notag\\
&=\left\{\left(\frac{\eta(\tau)}{\eta(3\tau)}\right)^{6}-3^3\left(\frac{\eta(3\tau)}{\eta(\tau)}\right)^{6}\right\}^2\notag\\
&=\prod_{k=0}^{5}\left(\frac{\eta(\tau)}{\eta(3\tau)}-\sqrt{3}\omega^k\frac{\eta(3\tau)}{\eta(\tau)}\right)^2.\notag
\end{align}
Hence, we have
$$\left|j_3^+(\tau)-j_3^+\left(\frac{1}{\sqrt{3}}e^{\frac{\pi i}{2}}\right)\right|=\prod_{k=0}^{5}\left|\frac{\eta(\tau)}{\eta(3\tau)}-\sqrt{3}\omega^k\frac{\eta(3\tau)}{\eta(\tau)}\right|^2.$$
Put
\begin{align}
\eta(\tau)&=q^{\frac{1}{24}}\sum_{n=0}^{\infty}\alpha_nq^n,\notag\\
\eta(3\tau)&=q^{\frac{1}{8}}\sum_{n=0}^{\infty}\beta_nq^n,\notag\\
\frac{\eta(\tau)}{\eta(3\tau)}&=q^{-\frac{1}{12}}\sum_{n=0}^{\infty}a_nq^n,\notag\\
\frac{\eta(3\tau)}{\eta(\tau)}&=q^{\frac{1}{12}}\sum_{n=0}^{\infty}b_nq^n,\notag
\end{align}
where $q=e^{2\pi i\tau}$. Then, Euler's pentagonal number theorem implies that $|\alpha_n|, |\beta_n|\le1$. We show that $|a_n|, |b_n|\le2^n$ for all $n\in\mathbb{Z}_{\ge0}$.\\
Since, $\alpha_n=\sum_{k=0}^{n}a_k\beta_{n-k}$, we have
$$|a_n|=|a_n\beta_0|=\left|-\sum_{k=0}^{n-1}a_k\beta_{n-k}+\alpha_n\right|\le\sum_{k=0}^{n-1}|a_k|+1.$$
Then, induction on $n$ implies $|a_n|\le2^n$. Similarly, we also have $|b_n|\le2^n$.\\
To calculate a lower bound for $j_3^+(\tau)-j_3^+\left(\frac{1}{\sqrt{3}}e^{\frac{\pi i}{2}}\right)$, we bound the derivative of $\frac{\eta(\tau)}{\eta(3\tau)}-\sqrt{3}\omega^k\frac{\eta(3\tau)}{\eta(\tau)}$ with respect to $x$ by
\begin{align}
&\left|\frac{d}{dx}\left(\frac{\eta(\tau)}{\eta(3\tau)}-\sqrt{3}\omega^k\frac{\eta(3\tau)}{\eta(\tau)}\right)\right|\notag\\
=&\left|\frac{d}{dx}\left(q^{-\frac{1}{12}}\sum_{n=0}^{\infty}a_nq^n-\sqrt{3}\omega^kq^{\frac{1}{12}}\sum_{n=0}^{\infty}b_nq^n\right)\right|\notag\\
=&\left|\sum_{n=0}^{\infty}2\pi i\left(n-\frac{1}{12}\right)a_nq^{n-\frac{1}{12}}-\sqrt{3}\omega^k\sum_{n=0}^{\infty}2\pi i\left(n+\frac{1}{12}\right)b_nq^{n+\frac{1}{12}}\right|\notag\\
\le&2\pi\sum_{n=0}^{\infty}\left(\left|n-\frac{1}{12}\right||a_n|e^{-0.7\pi (n-\frac{1}{12})}+\sqrt{3}\left(n+\frac{1}{12}\right)|b_n|e^{-0.7\pi (n+\frac{1}{12})}\right)\notag\\
=&2\pi\sum_{n=0}^{100}\left(\left|n-\frac{1}{12}\right||a_n|e^{-0.7\pi (n-\frac{1}{12})}+\sqrt{3}\left(n+\frac{1}{12}\right)|b_n|e^{-0.7\pi (n+\frac{1}{12})}\right)\notag\\
&\hspace{10pt}+2\pi\sum_{n=101}^{\infty}\left(\left|n-\frac{1}{12}\right||a_n|e^{-0.7\pi (n-\frac{1}{12})}+\sqrt{3}\left(n+\frac{1}{12}\right)|b_n|e^{-0.7\pi (n+\frac{1}{12})}\right)\notag\\
\le&2\pi\sum_{n=0}^{100}\left(\left|n-\frac{1}{12}\right||a_n|e^{-0.7\pi (n-\frac{1}{12})}+\sqrt{3}\left(n+\frac{1}{12}\right)|b_n|e^{-0.7\pi (n+\frac{1}{12})}\right)\notag\\
&\hspace{10pt}+2\pi\sum_{n=101}^{\infty}\left(\left|n-\frac{1}{12}\right|2^ne^{-0.7\pi (n-\frac{1}{12})}+\sqrt{3}\left(n+\frac{1}{12}\right)2^ne^{-0.7\pi (n+\frac{1}{12})}\right)\notag\\
<&4.0200.\notag
\end{align}
If we evaluate $|\frac{\eta(\tau)}{\eta(3\tau)}-\sqrt{3}\omega^k\frac{\eta(3\tau)}{\eta(\tau)}|$ at the points $\tau_0=\frac{n}{2000}+0.35i$ for $-1000\le n\le1000$, the spacing between the points is small enough that on the entire interval. When $|x-\frac{n}{2000}|\le\frac{1}{4000}$, the difference between $|\frac{\eta(\tau)}{\eta(3\tau)}-\sqrt{3}\omega^k\frac{\eta(3\tau)}{\eta(\tau)}|$ and its value at $\tau_0$ can not exceed $\sqrt{2}\times4.0200\times\frac{1}{4000}\le1.4213\times10^{-3}<|\frac{\eta(\tau_0)}{\eta(3\tau_0)}-\sqrt{3}\omega^k\frac{\eta(3\tau_0)}{\eta(\tau_0)}|$. Hence, we have
\begin{align}
\left|j_3^+(\tau)-j_3^+\left(\frac{1}{\sqrt{3}}e^{\frac{\pi i}{2}}\right)\right|&\ge\min_{-1000\le n\le1000}\prod_{k=0}^{5}\left|\left|\frac{\eta(\tau_0)}{\eta(3\tau_0)}-\sqrt{3}\omega^k\frac{\eta(3\tau_0)}{\eta(\tau_0)}\right|-1.4213\times10^{-3}\right|^2\notag\\
&>106.42886.\notag
\end{align}
Second, we estimate $|j_3^+(\frac{1}{\sqrt{3}}e^{\frac{\pi i}{2}})-j_3^+(z)|$.\\
Since $j_3^+(\frac{1}{\sqrt{3}}e^{\frac{\pi i}{2}})=66$, $j_3^+(\frac{1}{\sqrt{3}}e^{\frac{5\pi}{6}})=-42$, and $j_3^+(\frac{1}{\sqrt{3}}e^{i\theta})$ is real for $\theta\in[\frac{\pi}{2},\frac{5\pi}{6}]$, $j_3^+(\frac{1}{\sqrt{3}}e^{i\theta})$ must be monotonically increasing. Hence, we have
$$\left|j_3^+\left(\frac{1}{\sqrt{3}}e^{\frac{\pi i}{2}}\right)-j_3^+(z)\right|\le\left|j_3^+\left(\frac{1}{\sqrt{3}}e^{\frac{\pi i}{2}}\right)-j_3^+(\frac{1}{\sqrt{3}}e^{\frac{23}{10}i})\right|<106.01791$$
In conclusion,
$$|j_3^+(\tau)-j_3^+(z)|>106.42886-106.01791=0.41095.$$
\ \\

(d)\\
For $k,k_1,k_2\in\{0,4,6,8,10,14\}$, put
\begin{align}
E_k^+(z)&=\sum_{n=0}^{\infty}s_{k,n}q^n,\notag\\
E_{k_1}^+(z)E_{k_2}^+(z)&=\sum_{n=0}^{\infty}t_{k_1,k_2,n}q^n\notag.
\end{align}
Then,
 $$s_{k,n}=\begin{cases}1\ \ \ (n=0)\\-\frac{1}{1+3^{\frac{k}{2}}}\cdot\frac{2k}{B_k}\sigma_{k-1}(n)\ \ \ (n\notin3\mathbb{N})\\-\frac{1}{1+3^{\frac{k}{2}}}\cdot\frac{2k}{B_k}(\sigma_{k-1}(n)+3^{\frac{k}{2}}\sigma_{k-1}(\frac{n}{3}))\ \ \ (n\in3\mathbb{N}).\end{cases}$$
 Hence, we can get upper bounds of $s_{k,n}$ and $t_{k_1,k_2,n}$ as follows.
 $$|s_{k,n}|\le\frac{2k}{|B_k|}\sigma_{k-1}(n)+1\le\frac{2k}{|B_k|}n^k+1\le504n^k+1\le504(n+1)^k$$
 and
\begin{align}
|t_{k_1,k_2,n}|=\left|\sum_{\ell=0}^{n}s_{k_1,\ell}s_{k_2,n-\ell}\right|&\le\sum_{\ell=0}^{n}|s_{k_1,\ell}||s_{k_2,n-\ell}|\notag\\
&\le504^2\sum_{\ell=0}^{n}(\ell+1)^{k_1}(n-\ell+1)^{k_2}\notag\\
&\le504^2\sum_{\ell=0}^{n}(n+1)^{k_1+k_2}\notag\\
&=504^2(n+1)^{k_1+k_2+1}\notag.
\end{align}
Therefore, we can get an upper bound of $\Delta_{3,r_k}$ at z and $\tau$ as follows.
\begin{align}
|\Delta_{3,4}(z)|=|E_4^+(z)|&=\left|\sum_{n=0}^{\infty}s_{4,n}q^n\right|\notag\\
&\le\sum_{n=0}^{\infty}|s_{4,n}|e^{-2\pi n{\rm Im}(z)}\notag\\
&=\sum_{n=0}^{200}|s_{4,n}|e^{-2\pi n{\rm Im}(z)}+\sum_{n=201}^{\infty}|s_{4,n}|e^{-2\pi n{\rm Im}(z)}\notag\\
&\le\sum_{n=0}^{200}|s_{4,n}|e^{-2\pi n{\rm Im}(z)}+504\sum_{n=201}^{\infty}(n+1)^4e^{-2\pi n{\rm Im}(z)}\notag\\
&\le\sum_{n=0}^{200}|s_{4,n}|e^{-\frac{2\pi}{\sqrt{3}}\sin{(\frac{23}{10})}n}+504\sum_{n=201}^{\infty}(n+1)^4e^{-\frac{2\pi}{\sqrt{3}}\sin{(\frac{23}{10})}n}\notag\\
&<3.8757,\notag
\end{align}
\begin{align}
|\Delta_{3,4}(\tau)|&\le\sum_{n=0}^{200}|s_{4,n}|e^{-2\pi n{\rm Im}(\tau)}+504\sum_{n=201}^{\infty}(n+1)^4e^{-2\pi n{\rm Im}(\tau)}\notag\\
&=\sum_{n=0}^{200}|s_{4,n}|e^{-0.7\pi n}+504\sum_{n=201}^{\infty}(n+1)^4e^{-0.7\pi n}\notag\\
&<7.8622,\notag
\end{align}
\begin{align}
|\Delta_{3,6}(z)|=|E_6^+(z)|&=\left|\sum_{n=0}^{\infty}s_{6,n}q^n\right|\notag\\
&\le\sum_{n=0}^{\infty}|s_{6,n}|e^{-2\pi n{\rm Im}(z)}\notag\\
&\le\sum_{n=0}^{200}|s_{6,n}|e^{-2\pi n{\rm Im}(z)}+504\sum_{n=201}^{\infty}(n+1)^6e^{-2\pi n{\rm Im}(z)}\notag\\
&\le\sum_{n=0}^{200}|s_{6,n}|e^{-\frac{2\pi}{\sqrt{3}}\sin{(\frac{23}{10})}n}+504\sum_{n=201}^{\infty}(n+1)^6e^{-\frac{2\pi}{\sqrt{3}}\sin{(\frac{23}{10})}n}\notag\\
&<6.7891,\notag
\end{align}
\begin{align}
|\Delta_{3,6}(\tau)|&\le\sum_{n=0}^{200}|s_{6,n}|e^{-2\pi n{\rm Im}(\tau)}+504\sum_{n=201}^{\infty}(n+1)^6e^{-2\pi n{\rm Im}(\tau)}\notag\\
&=\sum_{n=0}^{200}|s_{6,n}|e^{-0.7\pi n}+504\sum_{n=201}^{\infty}(n+1)^6e^{-0.7\pi n}\notag\\
&<21.157,\notag
\end{align}
\begin{align}
|\Delta_{3,8}(z)|&=\frac{41}{1728}|E_4^+(z)^2-E_8^+(z)|\notag\\
&=\frac{41}{1728}\left|\sum_{n=0}^{\infty}(t_{4,4,n}-s_{8,n})q^n\right|\notag\\
&\le\frac{41}{1728}\left(\sum_{n=0}^{200}|t_{4,4,n}-s_{8,n}|e^{-2\pi n{\rm Im}(z)}+\sum_{n=201}^{\infty}(|t_{4,4,n}|+|s_{8,n}|)e^{-2\pi n{\rm Im}(z)}\right)\notag\\
&\le\frac{41}{1728}\left(\sum_{n=0}^{200}|t_{4,4,n}-s_{8,n}|e^{-2\pi n{\rm Im}(z)}+2\cdot504^2\sum_{n=201}^{\infty}(n+1)^9e^{-2\pi n{\rm Im}(z)}\right)\notag\\
&\le\frac{41}{1728}\left(\sum_{n=0}^{200}|t_{4,4,n}-s_{8,n}|e^{-\frac{2\pi}{\sqrt{3}}\sin{(\frac{23}{10})}n}+2\cdot504^2\sum_{n=201}^{\infty}(n+1)^9e^{-\frac{2\pi}{\sqrt{3}}\sin{(\frac{23}{10})}n}\right)\notag\\
&<0.10414,\notag
\end{align}
\begin{align}
|\Delta_{3,8}(\tau)|&\le\frac{41}{1728}\left(\sum_{n=0}^{200}|t_{4,4,n}-s_{8,n}|e^{-2\pi n{\rm Im}(\tau)}+2\cdot504^2\sum_{n=201}^{\infty}(n+1)^9e^{-2\pi n{\rm Im}(\tau)}\right)\notag\\
&\le\frac{41}{1728}\left(\sum_{n=0}^{200}|t_{4,4,n}-s_{8,n}|e^{-0.7\pi n}+2\cdot504^2\sum_{n=201}^{\infty}(n+1)^9e^{-0.7\pi n}\right)\notag\\
&<0.24233,\notag
\end{align}
\begin{align}
|\Delta_{3,10}(z)|&=\frac{61}{432}|E_4^+(z)E_6^+(z)-E_8^+(z)|\notag\\
&=\frac{61}{432}\left|\sum_{n=0}^{\infty}(t_{4,6,n}-s_{10,n})q^n\right|\notag\\
&\le\frac{61}{432}\left(\sum_{n=0}^{200}|t_{4,6,n}-s_{10,n}|e^{-2\pi n{\rm Im}(z)}+\sum_{n=201}^{\infty}(|t_{4,6,n}|+|s_{10,n}|)e^{-2\pi n{\rm Im}(z)}\right)\notag\\
&\le\frac{61}{432}\left(\sum_{n=0}^{200}|t_{4,6,n}-s_{10,n}|e^{-2\pi n{\rm Im}(z)}+2\cdot504^2\sum_{n=201}^{\infty}(n+1)^{11}e^{-2\pi n{\rm Im}(z)}\right)\notag\\
&\le\frac{61}{432}\left(\sum_{n=0}^{200}|t_{4,6,n}-s_{10,n}|e^{-\frac{2\pi}{\sqrt{3}}\sin{(\frac{23}{10})}n}+2\cdot504^2\sum_{n=201}^{\infty}(n+1)^{11}e^{-\frac{2\pi}{\sqrt{3}}\sin{(\frac{23}{10})}n}\right)\notag\\
&<0.26974,\notag
\end{align}
\begin{align}
|\Delta_{3,10}(\tau)|&\le\frac{61}{432}\left(\sum_{n=0}^{200}|t_{4,6,n}-s_{10,n}|e^{-2\pi n{\rm Im}(\tau)}+2\cdot504^2\sum_{n=201}^{\infty}(n+1)^{11}e^{-2\pi n{\rm Im}(\tau)}\right)\notag\\
&\le\frac{61}{432}\left(\sum_{n=0}^{200}|t_{4,6,n}-s_{10,n}|e^{-0.7\pi n}+2\cdot504^2\sum_{n=201}^{\infty}(n+1)^{11}e^{-0.7\pi n}\right)\notag\\
&<0.81140,\notag
\end{align}
\begin{align}
|\Delta_{3,14}(z)|&=\frac{22427}{272160}|E_6^+(z)E_8^+(z)-E_8^+(z)|\notag\\
&=\frac{22427}{272160}\left|\sum_{n=0}^{\infty}(t_{6,8,n}-s_{14,n})q^n\right|\notag\\
&\le\frac{22427}{272160}\left(\sum_{n=0}^{200}|t_{6,8,n}-s_{14,n}|e^{-2\pi n{\rm Im}(z)}+\sum_{n=201}^{\infty}(|t_{6,8,n}|+|s_{14,n}|)e^{-2\pi n{\rm Im}(z)}\right)\notag\\
&\le\frac{22427}{272160}\left(\sum_{n=0}^{200}|t_{6,8,n}-s_{14,n}|e^{-2\pi n{\rm Im}(z)}+2\cdot504^2\sum_{n=201}^{\infty}(n+1)^{15}e^{-2\pi n{\rm Im}(z)}\right)\notag\\
&\le\frac{22427}{272160}\left(\sum_{n=0}^{200}|t_{6,8,n}-s_{14,n}|e^{-\frac{2\pi}{\sqrt{3}}\sin{(\frac{23}{10})}n}+2\cdot504^2\sum_{n=201}^{\infty}(n+1)^{15}e^{-\frac{2\pi}{\sqrt{3}}\sin{(\frac{23}{10})}n}\right)\notag\\
&<0.54192,\notag
\end{align}
\begin{align}
|\Delta_{3,14}(\tau)|&\le\frac{22427}{272160}\left(\sum_{n=0}^{200}|t_{6,8,n}-s_{14,n}|e^{-2\pi n{\rm Im}(\tau)}+2\cdot504^2\sum_{n=201}^{\infty}(n+1)^{15}e^{-2\pi n{\rm Im}(\tau)}\right)\notag\\
&\le\frac{22427}{272160}\left(\sum_{n=0}^{200}|t_{6,8,n}-s_{14,n}|e^{-0.7\pi n}+2\cdot504^2\sum_{n=201}^{\infty}(n+1)^{15}e^{-0.7\pi n}\right)\notag\\
&<3.0481.\notag
\end{align}
Hence, we have
$$|\Delta_{3,r_k}(z)||\Delta_{3,14-r_k}(\tau)|<3.1448.$$
\end{proof}
\ \\
{\large(The proof of Proposition2.4(a))}
\begin{proof}
When $\ell_k\ge0$, the value of the right hand side of $(10)$ is bounded above by
\begin{align}
&e^{-2\pi (7\ell_k+18)(\frac{1}{\sqrt{3}}\sin{(\frac{23}{10})}-0.35)}\times \left(\frac{1.0258\times10^{-2}}{4.3086\times10^{-4}}\right)^{\ell_k}\times\frac{3.1448}{4.3086\times10^{-4}\times0.41095}\notag\\
&=\left(e^{-14\pi(\frac{1}{\sqrt{3}}\sin{(\frac{23}{10})}-0.35)}\frac{1.0258\times10^2}{4.3086\times10^{-4}}\right)^{\ell_k}\times e^{-36\pi(\frac{1}{\sqrt{3}}\sin{(\frac{23}{10})}-0.35)}\frac{3.1448}{4.3086\times10^{-4}\times0.41095}\notag\\
&<(0.68936)^{\ell_k}\times1.9674\notag\\
&\le1.9674.\notag
\end{align}
When $\ell_k<0$, the value of the right hand side of $(10)$ is bounded above by
\begin{align}
&e^{-2\pi (-11\ell_k+18)(\frac{1}{\sqrt{3}}\sin{(\frac{23}{10})}-0.35)}\times \left(\frac{5.0415\times10^{-2}}{2.8964\times10^{-4}}\right)^{-\ell_k}\times\frac{3.1448}{4.3086\times10^{-4}\times0.41095}\notag\\
&=\left(e^{-22\pi(\frac{1}{\sqrt{3}}\sin{(\frac{23}{10})}-0.35)}\frac{5.0415\times10^{-2}}{2.8964\times10^{-4}}\right)^{-\ell_k}\times e^{-36\pi(\frac{1}{\sqrt{3}}\sin{(\frac{23}{10})}-0.35)}\frac{3.1448}{4.3086\times10^{-4}\times0.41095}\notag\\
&<(0.66589)^{-\ell_k}\times1.9674\notag\\
&\le0.66589\times1.9674\notag\\
&<1.3101.\notag
\end{align}
We complete the ploof of Proposition 2.4(a).
\end{proof}

\section{The case of $\frac{23}{10}<\theta<\frac{5\pi}{6}-\frac{12}{25m}$}
Put $z=\frac{1}{\sqrt{3}}e^{i\theta}$ $(\frac{\pi}{2}<\theta<\frac{23}{10})$ and $A'=0.15$. We move the countor of integration given in $(4)$ downward to a height $A'$. As we do so, each pole $\tau_0$ of $G(\tau,z)$ in the region defined by
\begin{center}
$-\frac{1}{2}<{\rm Re}(\tau)<\frac{1}{2}$ and $A'<{\rm Im}(\tau)<A$
\end{center}
will contribute a term $2\pi i\cdot{\rm Res}_{\tau=\tau_0}G(\ \cdot\ ,z)$ to the equation. The pole of $G(\ \cdot\ ,z)$ occurs when $\tau=z, \frac{-1}{3z}, \frac{z}{3z+1}, \frac{-1}{3z+3}, \frac{-z-1}{3z+2}$ or $\frac{3z+1}{6z+3}$ which are equivalent to $z$ under the action of $\Gamma_0^+(3)$. Then the residue theorem yields
\begin{multline}
\int_{-\frac{1}{2}+iA'}^{\frac{1}{2}+iA'}G(\tau,z)d\tau=f_{k,m}(z)\\+2\pi i\left\{\sum_{\tau=z,\frac{-1}{3z}}{\rm Res}_\tau G(\ \cdot\ ,z)+\sum_{\tau=\frac{z}{3z+1},\frac{-1}{3z+3}}{\rm Res}_\tau G(\ \cdot\ ,z)+\sum_{\tau=\frac{-z-1}{3z+2},\frac{3z+1}{6z+3}}{\rm Res}_\tau G(\ \cdot\ ,z)\right\}.\notag
\end{multline}
By the same calculation as §4, we have
\begin{multline}
e^{-2\pi m\frac{1}{\sqrt{3}}\sin\theta}e^{\frac{ik\theta}{2}}f_{k,m}\left(\frac{1}{\sqrt{3}}e^{i\theta}\right)-2\cos\left(\frac{k\theta}{2}-2\pi m\frac{1}{\sqrt{3}}\cos\theta\right)\\
=e^{-2\pi m\frac{1}{\sqrt{3}}\sin\theta}e^{\frac{ik\theta}{2}}\int_{-\frac{1}{2}+iA'}^{\frac{1}{2}+iA'}G(\tau,z)d\tau+B_{k,m}(\theta)+C_{k,m}(\theta),
\end{multline}
where
\begin{align}
B_{k,m}(\theta):&=-e^{-2\pi m\frac{1}{\sqrt{3}}\sin{\theta}}e^{\frac{ik\theta}{2}}\cdot 2\pi i\sum_{\tau=\frac{z}{3z+1},\frac{-1}{3z+3}}{\rm Res}_\tau G(\ \cdot\ ,z)\notag\\
&=e^{-2\pi m\frac{1}{\sqrt{3}}(\sin{\theta}-\frac{\sin{\theta}}{4+2\sqrt{3}\cos{\theta}})}\{e^{i(\frac{k\theta}{2}-2\pi m\frac{\sqrt{3}+\cos{\theta}}{4\sqrt{3}+6\cos{\theta}})}(\sqrt{3}e^{i\theta}+1)^{-k}\notag\\
&\quad\quad+e^{-i(\frac{k\theta}{2}-2\pi m\frac{\sqrt{3}+\cos{\theta}}{4\sqrt{3}+6\cos{\theta}})}(\sqrt{3}e^{-i\theta}+1)^{-k}\},\notag
\end{align}
\begin{align}
C_{k,m}(\theta):&=-e^{-2\pi m\frac{1}{\sqrt{3}}\sin{\theta}}e^{\frac{ik\theta}{2}}\cdot 2\pi i\sum_{\tau=\frac{-z-1}{3z+2},\frac{3z+1}{6z+3}}{\rm Res}_\tau G(\ \cdot\ ,z)\notag\\
&=e^{-2\pi m\frac{1}{\sqrt{3}}(\sin{\theta}-\frac{\sin{\theta}}{7+4\sqrt{3}\cos{\theta}})}\{e^{i(\frac{k\theta}{2}-2\pi m\frac{3\sqrt{3}+5\cos{\theta}}{7\sqrt{3}+12\cos{\theta}})}(\sqrt{3}e^{i\theta}+2)^{-k}\notag\\
&\quad\quad+e^{-i(\frac{k\theta}{2}-2\pi m\frac{3\sqrt{3}+5\cos{\theta}}{7\sqrt{3}+12\cos{\theta}})}(\sqrt{3}e^{-i\theta}+2)^{-k}\}.\notag
\end{align}
Hence, the absolute value of the left hand side of (11) is bounded above by
\begin{align}
e^{-2\pi m\frac{1}{\sqrt{3}}\sin{\theta}}&\int_{-\frac{1}{2}+iA'}^{\frac{1}{2}+iA'}|G(\tau,z)|d|\tau|+|B_{k,m}(\theta)|+|C_{k,m}(\theta)|\notag\\ 
&\le e^{-2\pi m\frac{1}{\sqrt{3}}\sin{\theta}}\int_{-\frac{1}{2}+iA'}^{\frac{1}{2}+iA'}e^{-2\pi m{\rm Im}\tau}\frac{|\varDelta_3^+(z)|^{\ell_k}}{|\varDelta_3^+(\tau)|^{\ell_k}}\cdot\frac{|\varDelta_{3,r_k}(z)||\varDelta_{3,14-r_k}(\tau)|}{|\varDelta_3^+(\tau)||j_3^+(\tau)-j_3^+(z)|}d|\tau|\notag\\
&\hspace{20pt}+|B_{k,m}(\theta)|+|C_{k,m}(\theta)|\notag\\
&\le e^{-2\pi m(\frac{1}{2\sqrt{3}}-0,15)}\int_{-\frac{1}{2}+iA'}^{\frac{1}{2}+iA'}\frac{|\varDelta_3^+(z)|^{\ell_k}}{|\varDelta_3^+(\tau)|^{\ell_k}}\cdot\frac{|\varDelta_{3,r_k}(z)||\varDelta_{3,14-r_k}(\tau)|}{|\varDelta_3^+(\tau)||j_3^+(\tau)-j_3^+(z)|}d|\tau|\notag\\
&\hspace{20pt}+|B_{k,m}(\theta)|+|C_{k,m}(\theta)|.
\end{align}
To prove the Proposition 2.4 (b), we show the next lemma.
\begin{lem}
Let $z=\frac{1}{\sqrt{3}}e^{i \theta}$ with $\frac{23}{10}\le \theta <\frac{5\pi}{6}-\frac{12}{25m}$ and $\tau =x+0.15i$ with $-\frac{1}{2}\le x\le \frac{1}{2}$,\\
{\rm (a)} $3.4094\times10^{-4}<|\Delta_3^+(z)|<0.22521$.\\
{\rm (b)} $7.8764\times10^{-6}<|\Delta_3^+(\tau)|<61.432$.\\
{\rm (c)} $|j_3^+(\tau)-j_3^+(z)|>4.1403$.\\
{\rm (d)} For $r_k=0,4,6,8,10,14,\ |\Delta_{3,r_k}(z)||\Delta_{3,14-r_k}(\tau)|<1.8006\times10^3$.\\
{\rm (e)} $|B_{k,m}(\theta)|<0.62504$.\\
{\rm (f)} $|C_{k,m}(\theta)|<0.25843$.
\end{lem}
\begin{proof}
We can show (a),(b), and (d) by exactly the same way as the proof of Lemma 4.1. Hence, we only have to prove (c),(e), and (f).

(c)
$$|j_3^+(\tau)-j_3^+(z)|\ge\left|j_3^+(\tau)-j_3^+\left(\frac{1}{\sqrt{3}}e^{\frac{5\pi i}{6}}\right)\right|-\left|j_3^+\left(\frac{1}{\sqrt{3}}e^{\frac{5\pi i}{6}}\right)-j_3^+(z)\right|.$$
First, we estimate $|j_3^+(\tau)-j_3^+(\frac{1}{\sqrt{3}}e^{\frac{5\pi i}{6}})|$.\\
Put $\omega=e^{\frac{\pi i}{6}}$,
\begin{align}
j_3^+(\tau)-j_3^+\left(\frac{1}{\sqrt{3}}e^{\frac{5\pi i}{6}}\right)&=j_3^+(\tau)+42\notag\\
&=\left(\frac{\eta(\tau)}{\eta(3\tau)}\right)^{12}+3^6\left(\frac{\eta(3\tau)}{\eta(\tau)}\right)^{12}+54\notag\\
&=\left\{\left(\frac{\eta(\tau)}{\eta(3\tau)}\right)^{6}+3^3\left(\frac{\eta(3\tau)}{\eta(\tau)}\right)^{6}\right\}^2\notag\\
&=\prod_{k=0}^{5}\left(\frac{\eta(\tau)}{\eta(3\tau)}-\sqrt{3}\omega^{2k+1}\frac{\eta(3\tau)}{\eta(\tau)}\right)^2.\notag
\end{align}
Hence, we have
$$\left|j_3^+(\tau)-j_3^+\left(\frac{1}{\sqrt{3}}e^{\frac{5\pi i}{6}}\right)\right|=\prod_{k=0}^{5}\left|\frac{\eta(\tau)}{\eta(3\tau)}-\sqrt{3}\omega^{2k+1}\frac{\eta(3\tau)}{\eta(\tau)}\right|^2.$$
To calculate a lower bound for $j_3^+(\tau)-j_3^+\left(\frac{1}{\sqrt{3}}e^{\frac{\pi i}{2}}\right)$, we bound the derivative of $\frac{\eta(\tau)}{\eta(3\tau)}-\sqrt{3}\omega^{2k+1}\frac{\eta(3\tau)}{\eta(\tau)}$ with respect to $x$ by similar calculation in the proof of Lemma 4.1 (c) as follows.
\begin{align}
&\left|\frac{d}{dx}\left(\frac{\eta(\tau)}{\eta(3\tau)}-\sqrt{3}\omega^{2k+1}\frac{\eta(3\tau)}{\eta(\tau)}\right)\right|\notag\\
=&\left|\frac{d}{dx}\left(q^{-\frac{1}{12}}\sum_{n=0}^{\infty}a_nq^n-\sqrt{3}\omega^{2k+1}q^{\frac{1}{12}}\sum_{n=0}^{\infty}b_nq^n\right)\right|\notag\\
=&\left|\sum_{n=0}^{\infty}2\pi i\left(n-\frac{1}{12}\right)a_nq^{n-\frac{1}{12}}-\sqrt{3}\omega^k\sum_{n=0}^{\infty}2\pi i\left(n+\frac{1}{12}\right)b_nq^{n+\frac{1}{12}}\right|\notag\\
\le&2\pi\sum_{n=0}^{\infty}\left(\left|n-\frac{1}{12}\right||a_n|e^{-0.3\pi (n-\frac{1}{12})}+\sqrt{3}\left(n+\frac{1}{12}\right)|b_n|e^{-0.3\pi (n+\frac{1}{12})}\right)\notag\\
=&2\pi\sum_{n=0}^{100}\left(\left|n-\frac{1}{12}\right||a_n|e^{-0.3\pi (n-\frac{1}{12})}+\sqrt{3}\left(n+\frac{1}{12}\right)|b_n|e^{-0.3\pi (n+\frac{1}{12})}\right)\notag\\
&\hspace{10pt}+2\pi\sum_{n=101}^{\infty}\left(\left|n-\frac{1}{12}\right||a_n|e^{-0.3\pi (n-\frac{1}{12})}+\sqrt{3}\left(n+\frac{1}{12}\right)|b_n|e^{-0.3\pi (n+\frac{1}{12})}\right)\notag\\
\le&2\pi\sum_{n=0}^{100}\left(\left|n-\frac{1}{12}\right||a_n|e^{-0.3\pi (n-\frac{1}{12})}+\sqrt{3}\left(n+\frac{1}{12}\right)|b_n|e^{-0.3\pi (n+\frac{1}{12})}\right)\notag\\
&\hspace{10pt}+2\pi\sum_{n=101}^{\infty}\left(\left|n-\frac{1}{12}\right|2^ne^{-0.3\pi (n-\frac{1}{12})}+\sqrt{3}\left(n+\frac{1}{12}\right)2^ne^{-0.3\pi (n+\frac{1}{12})}\right)\notag\\
<&32.023.\notag
\end{align}
If we evaluate $|\frac{\eta(\tau)}{\eta(3\tau)}-\sqrt{3}\omega^{2k+1}\frac{\eta(3\tau)}{\eta(\tau)}|$ at the points $\tau_0=\frac{n}{2000}+0.15i$ for $-1000\le n\le1000$, the spacing between the points is small enough that on the entire interval. When $|x-\frac{n}{2000}|\le\frac{1}{4000}$, the difference between $|\frac{\eta(\tau)}{\eta(3\tau)}-\sqrt{3}\omega^{2k+1}\frac{\eta(3\tau)}{\eta(\tau)}|$ and its value at $\tau_0$ can not exceed $\sqrt{2}\times32.023\times\frac{1}{4000}\le1.1322\times10^{-2}<|\frac{\eta(\tau_0)}{\eta(3\tau_0)}-\sqrt{3}\omega^{2k+1}\frac{\eta(3\tau_0)}{\eta(\tau_0)}|$. Hence, we have
\begin{align} 
\left|j_3^+(\tau)-j_3^+\left(\frac{1}{\sqrt{3}}e^{\frac{5\pi i}{6}}\right)\right|&\ge\min_{-1000\le n\le1000}\prod_{k=0}^{5}\left|\left|\frac{\eta(\tau_0)}{\eta(3\tau_0)}-\sqrt{3}\omega^{2k+1}\frac{\eta(3\tau_0)}{\eta(\tau_0)}\right|-1.1322\times10^{-2}\right|^2\notag\\
&>6.1224.\notag
\end{align}
Second, we estimate $|j_3^+(\frac{1}{\sqrt{3}}e^{\frac{5\pi i}{6}})-j_3^+(z)|$.\\
Since $j_3^+(\frac{1}{\sqrt{3}}e^{\frac{\pi i}{2}})=66$, $j_3^+(\frac{1}{\sqrt{3}}e^{\frac{5\pi}{6}})=-42$, and $j_3^+(\frac{1}{\sqrt{3}}e^{i\theta})$ is real for $\theta\in[\frac{\pi}{2},\frac{5\pi}{6}]$, $j_3^+(\frac{1}{\sqrt{3}}e^{i\theta})$ must be monotonically increasing. Hence, we have
$$\left|j_3^+\left(\frac{1}{\sqrt{3}}e^{\frac{5\pi i}{6}}\right)-j_3^+(z)\right|\le\left|j_3^+\left(\frac{1}{\sqrt{3}}e^{\frac{5\pi i}{6}}\right)-j_3^+(\frac{1}{\sqrt{3}}e^{\frac{23}{10}i})\right|<1.9821$$
In conclusion,
$$|j_3^+(\tau)-j_3^+(z)|>6.1224-1.9821=4.1403.$$
\ \\

(e) and (f)\\
Recall that
\begin{align}
B_{k,m}(\theta)&=e^{-2\pi m\frac{1}{\sqrt{3}}(\sin{\theta}-\frac{\sin{\theta}}{4+2\sqrt{3}\cos{\theta}})}\{e^{i(\frac{k\theta}{2}-2\pi m\frac{\sqrt{3}+\cos{\theta}}{4\sqrt{3}+6\cos{\theta}})}(\sqrt{3}e^{i\theta}+1)^{-k}\notag\\
&\quad\quad+e^{-i(\frac{k\theta}{2}-2\pi m\frac{\sqrt{3}+\cos{\theta}}{4\sqrt{3}+6\cos{\theta}})}(\sqrt{3}e^{-i\theta}+1)^{-k}\},\notag\\
C_{k,m}(\theta)&=e^{-2\pi m\frac{1}{\sqrt{3}}(\sin{\theta}-\frac{\sin{\theta}}{7+4\sqrt{3}\cos{\theta}})}\{e^{i(\frac{k\theta}{2}-2\pi m\frac{3\sqrt{3}+5\cos{\theta}}{7\sqrt{3}+12\cos{\theta}})}(\sqrt{3}e^{i\theta}+2)^{-k}\notag\\
&\quad\quad+e^{-i(\frac{k\theta}{2}-2\pi m\frac{3\sqrt{3}+5\cos{\theta}}{7\sqrt{3}+12\cos{\theta}})}(\sqrt{3}e^{-i\theta}+2)^{-k}\}.\notag
\end{align}
Since $\frac{|k|}{2m}\le\frac{1}{3}$, we can bound above $B_{k,m}(\theta)$ and $C_{k,m}(\theta)$ by
\begin{align}
|B_{k,m}(\theta)|&\le e^{-2\pi m\frac{1}{\sqrt{3}}(\sin{\theta}-\frac{\sin{\theta}}{4+2\sqrt{3}\cos{\theta}})}\left\{|\sqrt{3}e^{i\theta}+1|^{|k|}+|\sqrt{3}e^{-i\theta}+1|^{|k|}\right\}\notag\\
&=2e^{-2\pi m\frac{1}{\sqrt{3}}(\sin{\theta}-\frac{\sin{\theta}}{4+2\sqrt{3}\cos{\theta}})}(4+2\sqrt{3}\cos{\theta})^{\frac{|k|}{2}}\notag\\
&=2\left\{e^{-2\pi\frac{1}{\sqrt{3}}(\sin{\theta}-\frac{\sin{\theta}}{4+2\sqrt{3}\cos{\theta}})}(4+2\sqrt{3}\cos{\theta})^{\frac{|k|}{2m}}\right\}^m\notag\\
&\le2\left\{e^{-2\pi\frac{1}{\sqrt{3}}(\sin{\theta}-\frac{\sin{\theta}}{4+2\sqrt{3}\cos{\theta}})}(4+2\sqrt{3}\cos{\theta})^{\frac{1}{3}}\right\}^m\notag
\end{align}
and
\begin{align}
|C_{k,m}(\theta)|&\le e^{-2\pi m\frac{1}{\sqrt{3}}(\sin{\theta}-\frac{\sin{\theta}}{7+4\sqrt{3}\cos{\theta}})}\left\{|\sqrt{3}e^{i\theta}+2|^{|k|}+|\sqrt{3}e^{-i\theta}+2|^{|k|}\right\}\notag\\
&=2e^{-2\pi m\frac{1}{\sqrt{3}}(\sin{\theta}-\frac{\sin{\theta}}{7+4\sqrt{3}\cos{\theta}})}(7+4\sqrt{3}\cos{\theta})^{\frac{|k|}{2}}\notag\\
&=2\left\{e^{-2\pi\frac{1}{\sqrt{3}}(\sin{\theta}-\frac{\sin{\theta}}{7+4\sqrt{3}\cos{\theta}})}(7+4\sqrt{3}\cos{\theta})^{\frac{|k|}{2m}}\right\}^m\notag\\
&\le2\left\{e^{-2\pi\frac{1}{\sqrt{3}}(\sin{\theta}-\frac{\sin{\theta}}{7+4\sqrt{3}\cos{\theta}})}(7+4\sqrt{3}\cos{\theta})^{\frac{1}{3}}\right\}^m.\notag
\end{align}
For $\theta\in[\frac{23}{10},\frac{5\pi}{6}]$, put
\begin{align}
g(\theta)&:=e^{-2\pi\frac{1}{\sqrt{3}}(\sin{\theta}-\frac{\sin{\theta}}{4+2\sqrt{3}\cos{\theta}})}(4+2\sqrt{3}\cos{\theta})^{\frac{1}{3}},\notag\\
h(\theta)&:=e^{-2\pi\frac{1}{\sqrt{3}}(\sin{\theta}-\frac{\sin{\theta}}{7+4\sqrt{3}\cos{\theta}})}(7+4\sqrt{3}\cos{\theta})^{\frac{1}{3}}.\notag
\end{align}
First, we show that $g'(\theta)>0$, $h'(\theta)>0$ for all $\theta\in(\frac{23}{10},\frac{5\pi}{6})$ and $g'(\theta)>2.4233$, $h'(\theta)>4.2632$ for all $\theta\in(\frac{5\pi}{6}-\frac{12}{575},\frac{5\pi}{6})$.
\begin{align}
g'(\theta)&=\frac{2}{\sqrt{3}}g(\theta)\times\left\{\frac{\sin{\theta}}{4+2\sqrt{3}\cos{\theta}}\left(\frac{2\sqrt{3}\pi \sin{\theta}}{4+2\sqrt{3}\cos{\theta}}-1\right)-\pi \cos{\theta}\cdot\frac{3+2\sqrt{3}\cos{\theta}}{4+2\sqrt{3}\cos{\theta}}\right\}\notag,\\
h'(\theta)&=\frac{2}{\sqrt{3}}h(\theta)\times\left\{\frac{2\sin{\theta}}{7+4\sqrt{3}\cos{\theta}}\left(\frac{2\sqrt{3}\pi \sin{\theta}}{7+4\sqrt{3}\cos{\theta}}-1\right)-\pi \cos{\theta}\cdot\frac{6+4\sqrt{3}\cos{\theta}}{7+4\sqrt{3}\cos{\theta}}\right\}.\notag\end{align}
When $\theta\in(\frac{23}{10},\frac{5\pi}{6})$, it is easy to show that $\frac{\sin{\theta}}{4+2\sqrt{3}\cos{\theta}}\left(\frac{2\sqrt{3}\pi \sin{\theta}}{4+2\sqrt{3}\cos{\theta}}-1\right)$, $\frac{2\sin{\theta}}{7+4\sqrt{3}\cos{\theta}}\left(\frac{2\sqrt{3}\pi \sin{\theta}}{7+4\sqrt{3}\cos{\theta}}-1\right)$ are positive monotonically increasing functions and $-\pi \cos{\theta}\cdot\frac{3+2\sqrt{3}\cos{\theta}}{4+2\sqrt{3}\cos{\theta}}$, $-\pi \cos{\theta}\cdot\frac{6+4\sqrt{3}\cos{\theta}}{7+4\sqrt{3}\cos{\theta}}$ are positive monotonically decreasing functions.\\
Hence, we get $g'(\theta)>0$, $h'(\theta)>0$. Moreover, when $\theta\in(\frac{5\pi}{6}-\frac{12}{575},\frac{5\pi}{6})$,
\begin{align}
g'(\theta)&>\frac{2}{\sqrt{3}}g\left(\frac{5\pi}{6}-\frac{12}{575}\right)\times\frac{\sin{(\frac{5\pi}{6}-\frac{12}{575}})}{4+2\sqrt{3}\cos{(\frac{5\pi}{6}-\frac{12}{575}})}\left(\frac{2\sqrt{3}\pi \sin{(\frac{5\pi}{6}-\frac{12}{575}})}{4+2\sqrt{3}\cos{(\frac{5\pi}{6}-\frac{12}{575}})}-1\right)\notag\\
&>2.4233,\notag\\
h'(\theta)&>\frac{2}{\sqrt{3}}h\left(\frac{5\pi}{6}-\frac{12}{575}\right)\times\frac{2\sin{(\frac{5\pi}{6}-\frac{12}{575}})}{7+4\sqrt{3}\cos{(\frac{5\pi}{6}-\frac{12}{575}})}\left(\frac{2\sqrt{3}\pi \sin{(\frac{5\pi}{6}-\frac{12}{575}})}{7+4\sqrt{3}\cos{(\frac{5\pi}{6}-\frac{12}{575}})}-1\right)\notag\\
&>4.2632.\notag
\end{align}
For all $\theta\in[\frac{23}{10},\frac{5\pi}{6}-\frac{12}{25m}]$, the mean value theorem implies that there exist $\theta_1, \theta_2\in(\frac{5\pi}{6}-\frac{12}{25m},\frac{5\pi}{6})$ such that
\begin{align}
g(\theta)\le g\left(\frac{5\pi}{6}-\frac{12}{25m}\right)&=g\left(\frac{5\pi}{6}\right)-g'(\theta_1)\cdot\frac{12}{25m}\notag\\
&=1-g'(\theta_1)\cdot\frac{12}{25m}\notag\\
&<1-\frac{1.1631}{m},\notag\\
h(\theta)\le h\left(\frac{5\pi}{6}-\frac{12}{25m}\right)&=h\left(\frac{5\pi}{6}\right)-h'(\theta_2)\cdot\frac{12}{25m}\notag\\
&=1-h'(\theta_2)\cdot\frac{12}{25m}\notag\\
&<1-\frac{2.0463}{m}.\notag
\end{align}
Therefore, we have 
\begin{align}
&|B_{k,m}(\theta)|\le2g(\theta)^{m}<2\left(1-\frac{1.1631}{m}\right)^m<2e^{-1.1631}<0.62504,\notag\\
&|C_{k,m}(\theta)|\le2g(\theta)^{m}<2\left(1-\frac{2.0463}{m}\right)^m<2e^{-2.0463}<0.25843.\notag
\end{align}
\end{proof}
\ \\
{\large(The proof of Proposition 2.4(b))}
\begin{proof}
When $\ell_k\ge0$, the value of the right hand side of $(12)$ is bounded above by
\begin{align}
&e^{-2\pi (18\ell_k+23)(\frac{1}{2\sqrt{3}}-0.15)}\times \left(\frac{0.22521}{7.8764\times10^{-6}}\right)^{\ell_k}\times\frac{1.8006\times10^3}{7.8764\times10^{-6}\times4.1403}+0.62504+0.25843\notag\\
&=\left(e^{-36\pi(\frac{1}{2\sqrt{3}}-0.15)}\frac{0.22521}{7.8764\times10^{-6}}\right)^{\ell_k}\times e^{-46\pi(\frac{1}{2\sqrt{3}}-0.15)}\frac{1.8006\times10^3}{7.8764\times10^{-6}\times4.1403}+0.88347\notag\\
&<(4.4145\times10^{-3})^{\ell_k}\times0.10931+0.88347\notag\\
&\le0.10931+0.88347\notag\\
&=0.99728.\notag
\end{align}
When $\ell_k<0$, the value of the right hand side of $(15)$ is bounded above by
\begin{align}
&e^{-2\pi (-18\ell_k+23)(\frac{1}{2\sqrt{3}}-0.15)}\times \left(\frac{61.432}{3.4094\times10^{-4}}\right)^{-\ell_k}\times\frac{1.8006\times10^3}{7.8764\times10^{-6}\times4.1403}+0.62504+0.25843\notag\\
&=\left(e^{-36\pi(\frac{1}{2\sqrt{3}}-0.15)}\frac{61.432}{3.4094\times10^{-4}}\right)^{-\ell_k}\times e^{-46\pi(\frac{1}{2\sqrt{3}}-0.15)}\frac{1.8006\times10^3}{7.8764\times10^{-6}\times4.1403}+0.88347\notag\\
&<(2.7819\times10^{-2})^{-\ell_k}\times0.10931+0.88347\notag\\
&\le2.7819\times10^{-2}\times0.10931+0.88347\notag\\
&<0.88652.\notag
\end{align}
We complete the ploof of Proposition 2.4(b).
\end{proof}

S. Hanamoto\\
Faculity of mathematics, kyushu university, Motooka 744, Nishi-ku Fukuoka 819-0395, Japan\\
E-mail: s-hanamoto@math.kyushu-u.ac.jp\\
\ \\
S. Kuga\\
Graduate school of mathematics, Kyushu university, Motooka 744, Nishi-ku Fukuoka 819-0395, Japan\\
E-mail: ma217058@math.kyushu-u.ac.jp
\end{document}